\documentclass[11pt]{amsart}
\usepackage{cases}
\usepackage{latexsym}
\usepackage[arrow,matrix]{xy}
\usepackage{stmaryrd}
\usepackage{amsfonts}
\usepackage{amsmath,amssymb,amscd,bbm,amsthm,mathrsfs,dsfont}
\usepackage{fancyhdr}
\usepackage{amsxtra,ifthen}
\usepackage{verbatim}

\numberwithin{equation}{section}

\theoremstyle{plain}
\newtheorem{prop}{Proposition}[section]
\newtheorem{thm}[prop]{Theorem}

\newtheorem{lem}[prop]{Lemma}
\newtheorem{dfn}[prop]{Definition}

\theoremstyle{definition}
\newtheorem{example}[prop]{Example}
\newtheorem{remark}[prop]{Remark}
\newtheorem{remarks}[prop]{Remarks}
\newtheorem{question}[prop]{Question}

\newcommand{\lam}{\lambda}

\begin{document}

\title[Nakayama twisted centers and dual bases]{Nakayama twisted centers and dual bases of Frobenius cellular algebras}

\thanks{Date: September 20, 2013}

\thanks {This work is supported by Fundamental Research Funds for the Central Universities
(N110423007)
and the Natural Science Foundation of Hebei Province, China
(A2013501055).}

\author{Yanbo Li}

\address{School of Mathematics and Statistics, Northeastern
University at Qinhuangdao, Qinhuangdao, 066004, P.R. China}

\email{liyanbo707@163.com}

\begin{abstract}
For a Frobenius cellular algebra, we prove that if the left
(right) dual basis of a cellular basis is again cellular, then the
algebra is symmetric. Moreover, some ideals of the center are
constructed by using the so-called Nakayama twisted center.
\end{abstract}



\maketitle

\section{Introduction}

Frobenius algebras are finite dimensional algebras over a field
that have a certain self-dual property. These algebras appear in
not only some branches of algebra, such as representation theory,
Hopf algebras, algebraic geometry and so on, but also topology,
geometry and coding theory, even in the work on the solutions of
Yang-Baxter equation. Symmetric algebras including group algebras
of finite groups are a large source examples. Furthermore, finite
dimensional Hopf algebras over fields are Frobenius algebras for
which the Nakayama automorphism has finite order.

\smallskip

Cellular algebras were introduced by Graham and Lehrer in 1996.
The theory of cellular algebras provides a linear method to study
the representation theory of many interesting algebras. The
classical examples of cellular algebras include Hecke algebras of
finite type, Ariki-Koike algebras, q-Schur algebras, Brauer
algebras, partition algebras, Birman-Wenzl algebras and so on. We
refer the reader to \cite{G, GL, Xi1, Xi2} for details. It is
helpful to point out that some cellular algebras are symmetric,
including Hecke algebras of finite type, Ariki-Koike algebras
satisfying certain conditions et cetera. In \cite{L1, L3, LX}, Li
and Xiao studied the general theory of symmetric cellular
algebras, such as dual bases, cell modules, centers and radicals.

\smallskip

Apart from its own importance, Frobenius algebras are useful
because the symmetricity of an algebra is usually not easy to
verify. Thus it is meaningful to develop methods that can be used
to deal with algebras only known to be Frobenius. In \cite{L2,
LZ}, Li and Zhao investigated Nakayama automorphisms and
projective cell modules of Frobenius cellular algebras. In this
paper, we will study dual bases and central ideals. Exactly, we
prove that there not exists a non-symmetric Frobenius cellular
algebra such that the right or left dual basis being cellular.
Moreover, some ideals of the center are also constructed by using
the so-called Nakayama twisted center.

\smallskip

The paper is organized as follows. We first study Higman ideals
and Nakayama twisted centers of Frobenius algebras. Then in
Section 3, we detect properties of dual bases and construct ideals
of centers for Frobenius cellular algebras. In Section 4, we
illustrate some examples of non-symmetric Frobenius cellular
algebras and give some remarks.

\bigskip

\section{Higman ideals and Nakayama twisted centers}

Let $K$ be a field and let $A$ be a finite dimensional
$K$-algebra. Let $f:A\times A\rightarrow K$ be a $K$-bilinear
form. We say that $f$ is
 non-degenerate if the determinant of the matrix
$(f(a_{i},a_{j}))_{a_{i},a_{j}\in B}$ is not zero for some basis
$B$ of $A$. We say $f$ is associative if $f(ab,c)=f(a,bc)$ for all
$a,b,c\in A$, and symmetric if $f(a,b)=f(b,a)$ for all $a,b\in A$.
\begin{dfn}\label{2.1}
A $K$-algebra $A$ is called a Frobenius algebra if there is a
non-degenerate associative bilinear form $f$ on $A$. We call $A$ a
symmetric algebra if in addition $f$ is symmetric.
\end{dfn}

It is well known that for a Frobenius algebra $A$, there is an
automorphism $\alpha$ of $A$ such that
$Hom_K(A,K)\,\simeq\,_{\alpha}A_1$ as $A$-$A$-bimodules. This
automorphism is unique up to inner automorphisms and is called the
Nakayama automorphism. In \cite{HZ}, Holm and Zimmermann proved
the following lemma.

\begin{lem}{\rm\cite[Lemma 2.7]{HZ}}\label{2.2}
Let $A$ be a finite dimensional Frobenius algebra. Then an
automorphism $\alpha$ of $A$ is a Nakayama automorphism if and
only if $$f(a, b)=f(\alpha(b), a).$$
\end{lem}

Let $A$ be a Frobenius algebra with a basis $B=\{a_{i}\mid
i=1,\ldots,n\}$ and $f$ a non-degenerate associative bilinear form
$f$. Define a $K$-linear map $\tau: A\rightarrow R$ by
$\tau(a)=f(a,1)$. We call $\tau$ a symmetrizing trace if $A$ is
symmetric. Denote by $d=\{d_{i}\mid i=1,\ldots,n\}$ the basis
which is uniquely determined by the requirement that
$\tau(a_{i}d_{j})=\delta_{ij}$ and $D=\{D_{i}\mid i=i,\ldots,n\}$
the basis determined by the requirement that
$\tau(D_{j}a_{i})=\delta_{ij}$ for all $i, j=1,\ldots,n$. We will
call $d$ the right dual basis of $B$ and $D$ the left dual basis
of $B$, respectively. Then we can define an $R$-linear map
$\alpha:A\rightarrow A$ by $\alpha(d_{i})=D_{i}$. It is easy to
show that $\alpha$ is a Nakayama automorphism of $A$ by Lemma
\ref{2.2}. If $A$ is a symmetric algebra, then $\alpha$ is the
identity map and then the left and the right dual basis are the
same.

Write $a_{i}a_{j}=\sum_{k}r_{ijk}a_{k}$, where $r_{ijk}\in K$. We
proved the following lemma in \cite{L2}, which is useful in this
section.
\begin{lem}\cite[Lemma 1.3]{L2}\label{2.3}
Let $A$ be a Frobenius algebra with a basis $B$ and dual bases $d$
and $D$. Then the following hold:

\noindent{\rm(1)}\,\,$a_{i}d_{j}=\sum_{k}r_{kij}d_{k}$;

\noindent{\rm(2)}\,\,$D_{i}a_{j}=\sum_{k}r_{jki}D_{k}.$
\end{lem}

It is well known that $\{\sum\limits_{i}d_{i}aa_{i}\mid a\in A\}$
is an ideal of $Z(A)$, see \cite{CR}. The ideal
$H(A)=\{\sum\limits_{i}d_{i}aa_{i}\mid a\in A\}$ is called Higman
ideal of $Z(A)$. Note that $H(A)$ is independent of the choice of
$\tau$ and the dual bases $a_i$ and $d_i$. Then it is clear that
$H(A)=\{\sum\limits_{i}a_{i}aD_{i}\mid a\in A\}$.

A natural problem is to consider sets
$\{\sum\limits_{i}a_{i}ad_{i}\mid a\in A\}$ and
$\{\sum\limits_{i}D_{i}aa_{i}\mid a\in A\}$. We claim that these
two sets are the same. The proof is similar to that $H(A)$ is
independent of the choice of the dual bases, which can be found in
\cite{CR}.  We omit the details and left it to the reader. From
now on, let us write $\{\sum\limits_{i}a_{i}ad_{i}\mid a\in A\}$
by $H_{\alpha}(A)$.

In \cite{HZ} Holm and Zimmermann introduced Nakayama twisted
center $Z_{\alpha}(A)$ for a Frobenius algebra $A$. We will prove
that $H_{\alpha}(A)\subset Z_{\alpha}(A).$ Firstly, let us recall
the definition of Nakayama twisted centers of Frobenius algebras.

\begin{dfn}{\rm\cite[Definition 2.2]{HZ}}\label{2.4}
Let $A$ be a Frobenius algebra  and $\alpha$ a Nakayama
automorphism. The Nakayama twisted center is defined to be
$$Z_{\alpha}(A) := \{x\in A \mid xa = \alpha(a)x\,\,\, {\rm for \,\,\,all}\,\, a\in
A\}.$$
\end{dfn}

It is easy to know that $$Z_{\alpha}(A) = \{x\in A \mid
x\alpha^{-1}(a) = ax\,\,\, {\rm for\,\,\, all}\,\, a\in A\}.$$
Moreover, we denote by $$Z_{\alpha^{-1}}(A) = \{x\in A \mid xa =
\alpha^{-1}(a)x\,\,\, {\rm for\,\,\, all}\,\, a\in A\}.$$

For the Nakayama twisted center of a Frobenius algebra, we have
the following result:

\begin{lem}\label{2.5}
Let $A$ be a Frobenius algebra and let $\alpha$ be an automorphism
of $A$. Then for any $x,y\in Z_{\alpha}(A)$, we have\\
\noindent{\rm(1)}\,\,$\alpha(x)\in Z_{\alpha}(A)$.\\
\noindent{\rm(2)}\,\,$\alpha(xy)=xy$.\\
\noindent{\rm(3)}\,\,$Z_{\alpha^{-1}}(A)Z_{\alpha}(A)$ and
$Z_{\alpha}(A)Z_{\alpha^{-1}}(A)$ are both ideals of $Z(A)$.
\end{lem}

\begin{proof}
(1)\,We need to prove that $\alpha(x)a=\alpha(a)\alpha(x)$ for any
$a\in A$. In fact, since $\alpha$ is an automorphism of $A$,
$\alpha(x)a=\alpha(x)\alpha(\alpha^{-1}(a))=\alpha(x\alpha^{-1}(a))$.
Note that $x\in Z_{\alpha}(A)$, then
$\alpha(x\alpha^{-1}(a))=\alpha(ax)=\alpha(a)\alpha(x).$

(2)\,It follows from $x\in Z_{\alpha}(A)$ that $xy=\alpha(y)x$. By
(1), $y\in Z_{\alpha}(A)$ implies that $\alpha(y)\in
Z_{\alpha}(A)$. Then $\alpha(y)x=\alpha(x)\alpha(y)=\alpha(xy)$.

(3)\,$Z_{\alpha^{-1}}(A)Z_{\alpha}(A)\subset Z(A$ is clear by
definitions of $Z_{\alpha^{-1}}(A)$ and $Z_{\alpha}(A)$. Moreover,
$cZ_{\alpha}(A)\subset Z_{\alpha}(A)$ for all $c\in Z(A)$.
Consequently $Z_{\alpha^{-1}}(A)Z_{\alpha}(A)$ is an ideal of
$Z(A)$. It is proved similarly for
$Z_{\alpha}(A)Z_{\alpha^{-1}}(A)$.
\end{proof}

Now let us show that $H_{\alpha}(A)$ is in the Nakayama twisted
center of $A$.
\begin{lem}\label{2.6}
Let $A$ be a Frobenius algebra. Then $H_{\alpha}(A)\subset
Z_{\alpha}(A).$
\end{lem}

\begin{proof}\,For arbitrary $a\in A$, denote $\sum\limits_{i}a_{i}ad_{i}$ by $x_{a}$.
We prove $x_{a}\alpha^{-1}(b)=bx_{a}$ for all $b\in A$. In fact,
for an $a_{j}\in B$, we have from Lemma \ref{2.3} that
$$x_{a}\alpha^{-1}(a_j)=\sum_{i}a_{i}ad_{i}\alpha^{-1}(a_{j})=\sum_{i}a_{i}a\alpha^{-1}(D_{i}a_{j})
=\sum_{i,k}r_{jki}a_{i}ad_{k},$$ and
$$a_jx_{a}=\sum_{i}a_{j}a_{i}ad_{i}=\sum_{i,k}r_{jik}a_{k}ad_{i}.$$
Obviously, the right sides of the above two equations are equal.
\end{proof}

\begin{question}
It is well known that a Frobenius algebra $A$ is separable if and
only of $H(A)=Z(A)$. Then what implies that $H_{\alpha}(A)=
Z_{\alpha}(A)$?
\end{question}

In order to construct ideals of $Z(A)$, we hope to find some
elements of $Z_{\alpha^{-1}}(A)$ in view of Lemma \ref{2.5}. One
option may be the set $$\{\sum\limits_{i}d_{i}a\alpha(a_{i})\mid
a\in A\},$$ which will be denoted by $H_{\alpha^{-1}}(A)$.
\begin{lem}\label{2.7}
$H_{\alpha^{-1}}(A)\subset Z_{\alpha^{-1}}(A).$
\end{lem}
\begin{proof}
Denote $\sum\limits_{i}d_{i}a\alpha(a_{i})$ by $x_{\alpha(a)}$ for
some $a\in A$. Now let us prove that
$x_{\alpha(a)}\alpha(b)=bx_{\alpha(a)}$ for any $b\in A$. In fact,
for an $a_j\in B$, we have from Lemma \ref{2.3} that
$$x_{\alpha(a)}\alpha(a_j)=\sum\limits_{i}d_{i}a\alpha(a_{i})\alpha(a_j)
=\sum\limits_{i}d_{i}a\alpha(a_{i}a_j)=\sum_{i,k}r_{ijk}d_{i}a\alpha(a_{k}).$$
$$a_jx_{\alpha(a)}=\sum\limits_{i}a_jd_{i}a\alpha(a_{i})=\sum_{i,k}r_{kji}d_{k}a\alpha(a_{i}).$$
The right sides of the above two equalities are the same and then we complete the proof.
\end{proof}

The following lemma reveals some relation among $H(A)$,
$H_{\alpha}(A)$, $H_{\alpha^{-1}}(A)$, $Z_{\alpha}(A)$ and
$Z_{\alpha^{-1}}(A)$.
\begin{lem}\label{2.8}
The sets $H_{\alpha^{-1}}(A)Z_{\alpha}(A)$,
$Z_{\alpha}(A)H_{\alpha^{-1}}(A)$,
$H_{\alpha}(A)Z_{\alpha^{-1}}(A)$, and
$Z_{\alpha^{-1}}(A)H_{\alpha}(A)$,
 which
contained in $H(A)$ are all ideals of $Z(A)$.
\end{lem}

\begin{proof}
Firstly, let us consider $H_{\alpha^{-1}}(A)Z_{\alpha}(A)$. It
follows from Lemma \ref{2.5}, \ref{2.6} and \ref{2.7} that
$H_{\alpha^{-1}}(A)Z_{\alpha}(A)\subset Z(A)$. It is easy to know
that for any $c\in Z(A)$, $cZ_{\alpha}(A)\in Z_{\alpha}(A)$. This
implies that $H_{\alpha^{-1}}(A)Z_{\alpha}(A)$ is an ideal of
$Z(A)$. We now prove that
$H_{\alpha^{-1}}(A)Z_{\alpha}(A)\subseteq H(A).$ In fact, for any
$a\in A$ and $z\in Z_{\alpha}(A)$, we have
\begin{eqnarray*}
x_{\alpha(a)}z
&=&\sum_{i}d_{i}a\alpha(a_{i})z=\sum_{i}d_{i}a(\alpha(a_{i})z)\\
&=&\sum_{i}d_{i}a(za_i)=\sum_{i}d_{i}(az)a_i\in H(A).
\end{eqnarray*}

It is proved similarly for other sets.
\end{proof}

\begin{remark}
A natural generalization of $H_{\alpha^{-1}}(A)$ is
$$\{\sum_{i}d_{i}a\alpha^{m}(a_{i})\mid a\in A\}.$$ It is easy to
check that for all $b\in A$,
$$\sum_{i}d_{i}a\alpha^{m}(a_{i})\alpha^{m}(b)=b\sum_{i}d_{i}a\alpha^{m}(a_{i}).$$
\end{remark}

\bigskip

\section{Dual bases of Frobenius cellular algebras}

Cellular algebras were introduced by Graham and Lehrer in 1996.
The definition of a cellular algebra is as follows.
\begin{dfn} {\rm \cite{GL}}\label{3.1}
Let $R$ be a commutative ring with identity. An associative unital
$R$-algebra is called a cellular algebra with cell datum
$(\Lambda, M, C, i)$ if the following conditions are satisfied:
\begin{enumerate}
\item[(C1)] The finite set $\Lambda$ is a poset. Associated with
each $\lam\in\Lambda$, there is a finite set $M(\lam)$. The
algebra $A$ has an $R$-basis $\{C_{S,T}^\lam \mid S,T\in
M(\lam),\lam\in\Lambda\}$.

\item[(C2)] The map $i$ is an $R$-linear anti-automorphism of $A$
with $i^{2}=id$ which sends $C_{S,T}^\lam$ to $C_{T,S}^\lam $.

\item[(C3)] If $\lam\in\Lambda$ and $S,T\in M(\lam)$, then for any
element $a\in A$, we have\\
$$aC_{S,T}^\lam\equiv\sum_{S^{'}\in
M(\lam)}r_{a}(S',S)C_{S^{'},T}^{\lam} \,\,\,\,(\mod
 A(<\lam)),$$
where $r_{a}(S^{'},S)\in R$ is independent of $T$ and where
$A(<\lam)$ is the $R$-submodule of $A$ generated by
$\{C_{S'',T''}^\mu \mid \mu<\lam,  S'',T''\in M(\mu)\}$.
\end{enumerate}
\end{dfn}

If we apply $i$ to the equation in {\rm(C3)}, we obtain

{$\rm(C3')$} $C_{T,S}^\lam i(a)\equiv\sum_{S^{'}\in
M(\lam)}r_{a}(S^{'},S)C_{T,S^{'}}^{\lam} \,\,\,\,(\mod A(<\lam)).$

It is easy to check that
$$C_{S,T}^\lam C_{U,V}^\lam \equiv\Phi(T,U)C_{S,V}^\lam\,\,\,\, (\rm mod\,\,\, A(<\lam)),$$
where $\Phi(T,U)\in R$ depends only on $T$ and $U$. We say
$\lam\in\Lambda_0$ if there exist $S,T\in M(\lam)$ such that
$\Phi(S,T)\neq 0$.

Let $A$ be a finite dimensional Frobenius cellular $K$-algebra
with a cell datum $(\Lambda, M, C, i)$. Given a non-degenerate
bilinear form $f$,  denote the left dual basis by
$D=\{D_{S,T}^\lam \mid S,T\in M(\lam),\lam\in\Lambda\}$, which
satisfies
$$
\tau(D_{U,V}^{\mu}C_{S,T}^{\lam})=\delta_{\lam\mu}\delta_{SV}\delta_{TU}.
$$
Denote the right dual basis by $d=\{d_{S,T}^\lam \mid S,T\in
M(\lam),\lam\in\Lambda\}$, which satisfies
$$
\tau(C_{S,T}^{\lam}d_{U,V}^{\mu})=\delta_{\lam\mu}\delta_{S,V}\delta_{T,U}.
$$
For $\mu\in\Lambda$, let $A_{d}(>\mu)$ be the $K$-subspace of $A$
generated by
$$\{d_{X,Y}^\epsilon \mid X,Y\in M(\epsilon),\mu<\epsilon\}$$ and let $A_{D}(>\mu)$
be the $K$-subspace of $A$ generated by
$$\{D_{X,Y}^\epsilon \mid X,Y\in M(\epsilon),\mu<\epsilon\}.$$
Note that the $K$-linear map $\alpha$ which sends
$d_{X,Y}^\epsilon$ to $D_{X,Y}^\epsilon$ is a Nakayama
automorphism of the algebra $A$. The following result is easy, we
omit the proof.

\begin{lem}
Keep notations as above. Then $\{\alpha(C_{S,T}^{\lam})\mid
\lam\in\Lambda, S,T\in M(\lam)\}$ is a cellular basis if and only
if $i\alpha=\alpha i$.
\end{lem}

We obtained the following lemma in \cite{LZ} about structure
constants. It will play an important role in this paper.

\begin{lem}\label{3.2}
For arbitrary $\lam,\mu\in\Lambda$ and $S,T,P,Q\in M(\lam)$,
$U,V\in M(\mu)$ and $a\in A$, the following hold:
\begin{enumerate}
\item[(1)] $D_{U,V}^{\mu}C_{S,T}^{\lam}=\sum\limits_{\epsilon\in
\Lambda, X,Y\in
M(\epsilon)}r_{(S,T,\lam),(Y,X,\epsilon),(V,U,\mu)}D_{X,Y}^{\epsilon}.$
\item[(2)] $D_{U,V}^{\mu}C_{S,T}^{\lam}=\sum\limits_{\epsilon\in
\Lambda, X,Y\in
M(\epsilon)}R_{(Y,X,\epsilon),(U,V,\mu),(T,S,\lam)}C_{X,Y}^{\epsilon}.$
\item[(3)] $aD_{U,V}^{\mu}\equiv \sum\limits_{U'\in
M(\mu)}r_{i(\alpha^{-1}(a))}(U,U')D_{U',V}^{\mu}\,\,\,(\mod
A_{D}(>\mu))$.
\smallskip
\item[(4)] $D_{P,Q}^{\lam}C_{S,T}^{\lam}=0\,\,\,\, if \,\,\,Q\neq
S.$
\smallskip
\item[(5)] $D_{U,V}^{\mu}C_{S,T}^{\lam}=0 \,\,\,\,if
\,\,\,\mu\nleq \lam.$
\smallskip
\item[(6)]
$D_{T,S}^{\lam}C_{S,Q}^{\lam}=D_{T,P}^{\lam}C_{P,Q}^{\lam}.$
\smallskip
\item[(7)] $C_{S,T}^{\lam}d_{U,V}^{\mu}=\sum\limits_{\epsilon\in
\Lambda, X,Y\in
M(\epsilon)}r_{(Y,X,\epsilon),(S,T,\lam),(V,U,\mu)}d_{X,Y}^{\epsilon}.$
\item[(8)] $C_{S,T}^{\lam}d_{U,V}^{\mu}=\sum\limits_{\epsilon\in
\Lambda, X,Y\in
M(\epsilon)}R_{(U,V,\mu),(Y,X,\epsilon),(T,S,\lam)}C_{X,Y}^{\epsilon}.$
\item[(9)] $d_{U,V}^{\mu}a\equiv \sum\limits_{V'\in
M(\mu)}r_{\alpha(a)}(V,V')d_{U,V'}^{\mu}\,\,\,\,\,\,\,\,\,\,\,(\mod
A_{d}(>\mu))$.
\smallskip
\item[(10)] $C_{S,T}^{\lam}d_{P,Q}^{\lam}=0\,\, if \,\,T\neq P.$
\smallskip
\item[(11)] $C_{S,T}^{\lam}d_{U,V}^{\mu}=0 \,\,\,\,if\,\,\,
\mu\nleq \lam.$
\smallskip
\item[(12)]
$C_{S,T}^{\lam}d_{T,P}^{\lam}=C_{S,Q}^{\lam}d_{Q,P}^{\lam}.$
\end{enumerate}
\end{lem}

\begin{remark}\label{3.3}
We have from Lemma \ref{3.2} that if
$i(d_{S,T}^{\lam})=d_{T,S}^{\lam}$ for arbitrary $\lam\in\Lambda$
and $S,T\in M(\lam)$, then $\{d_{S,T}^{\lam}\mid\lam\in\Lambda,
S,T\in M(\lam)\}$ is again cellular with respect to the opposite
order on $\Lambda$. The similar claim also holds for the left dual
basis.
\end{remark}

In \cite{L2}, we obtained a result about the relationship among
$i$, $\tau$ and $\alpha$. For the convenience of the reader, we
copy it here.
\begin{lem}\cite[Theorem 2.5]{L2}\label{3.4}
Let $A$ be a Frobenius cellular algebra with cell datum $(\Lambda,
M, C, i)$. Then one of the following three statements holds
implies that the other two are equivalent:

\begin{enumerate}
\item[(1)] $i(d_{S,T}^{\lam})=d_{T,S}^{\lam}$ and
$i(D_{S,T}^{\lam})=D_{T,S}^{\lam}$ for arbitrary $\lam\in\Lambda$
and $S,T\in M(\lam)$.

\item[(2)] $\alpha=id$, that is, $A$ is a symmetric algebra.

\item[(3)] $\tau(a)=\tau(i(a))$ for all $a\in A$.
\end{enumerate}
\end{lem}

\begin{remark}\label{3.5}
In the proof of $(1),(3)\Rightarrow (2)$ in \cite{L2}, we only
need that $i(D_{S,T}^{\lam})=D_{T,S}^{\lam}$ or
$i(d_{S,T}^{\lam})=d_{T,S}^{\lam}$ for arbitrary $\lam\in\Lambda$
and $S,T\in M(\lam)$.
\end{remark}

The results of Lemma \ref{3.4} can be partially developed as
follows.

\begin{lem}\label{3.6}
If $i(D_{S,T}^{\lam})=D_{T,S}^{\lam}$ for arbitrary
$\lam\in\Lambda$ and $S,T\in M(\lam)$, then $\tau(a)=\tau(i(a))$
for all $a\in A$.
\end{lem}

\begin{proof}
Suppose that $\tau(a)\neq\tau(i(a))$ for some $a\in A$. Then there
exists $C_{S,T}^{\lam}$ such that
$\tau(C_{S,T}^{\lam})\neq\tau(C_{T,S}^{\lam})$. Let
$1=\sum_{\epsilon\in \lambda,\,\, X,Y\in
M(\epsilon)}r_{X,Y,\epsilon}D_{X,Y}^{\epsilon}$. Applying
$\alpha^{-1}$ on both sides yields that $1=\sum_{\epsilon\in
\lambda,\,\, X,Y\in
M(\epsilon)}r_{X,Y,\epsilon}d_{X,Y}^{\epsilon}$. Then
$$\tau(C_{S,T}^{\lam})=\tau(\sum_{\epsilon\in \lambda,\,\, X,Y\in
M(\epsilon)}r_{X,Y,\epsilon}D_{X,Y}^{\epsilon}C_{S,T}^{\lam})=r_{T,S,\lam},$$
$$\tau(C_{T,S}^{\lam})=\tau(C_{T,S}^{\lam}\sum_{\epsilon\in \lambda,\,\, X,Y\in
M(\epsilon)}r_{X,Y,\epsilon}d_{X,Y}^{\epsilon})=r_{S,T,\lam}.$$
This implies that $r_{T,S,\lam}\neq r_{S,T,\lam}$.

On the other hand, it follows from
$i(D_{S,T}^{\lam})=D_{T,S}^{\lam}$ that $$1=i(1)=\sum_{\epsilon\in
\lambda,\,\, X,Y\in
M(\epsilon)}r_{Y,X,\epsilon}D_{X,Y}^{\epsilon}.$$ This forces
$r_{X,Y,\epsilon}=r_{Y,X,\epsilon}$ for all $\epsilon\in \Lambda$,
$X,Y\in M(\epsilon)$. It is a contradiction. This implies that
$\tau(a)=\tau(i(a))$ for all $a\in A$.
\end{proof}

Combining Remarks \ref{3.3}, \ref{3.5} with Lemma \ref{3.6} yields
one main result of this section.
\begin{thm}\label{3.7}
Let $A$ be a finite dimensional Frobenius cellular algebra with
cell datum $(\Lambda, M, C, i)$. If the right (left) dual basis is
again cellular, then $A$ is symmetric.
\end{thm}

\begin{remark}
In symmetric case, the condition $\tau(a)=\tau(i(a))$ is
equivalent to that the dual basis being cellular. However, in
Frobenius case, $\tau(a)=\tau(i(a))$ does not implies the
cellularity of dual bases. We will give a counterexample in
Section 4.
\end{remark}

\medskip

Now let us study Nakayama twisted centers of Frobenius cellular
algebras by using dual bases.

Let $A$ be a finite dimensional Frobenius cellular algebra. For
arbitrary $\lam, \,\mu\in \Lambda$, $S,\,T\in M(\lam)$, $U,\,V\in
M(\mu)$, write
$$C_{S,T}^{\lam}C_{U,V}^{\mu}=\sum\limits_{\epsilon\in\Lambda,X,Y\in M(\epsilon)}
r_{(S,T,\lam),(U,V,\mu),(X,Y,\epsilon)}C_{X,Y}^{\epsilon},$$

$$D_{S,T}^{\lam}D_{U,V}^{\mu}=\sum\limits_{\epsilon\in\Lambda,X,Y\in M(\epsilon)}
R_{(S,T,\lam),(U,V,\mu),(X,Y,\epsilon)}D_{X,Y}^{\epsilon}.$$
Applying $\alpha^{-1}$ on both sides of the above equation, we
obtain
$$d_{S,T}^{\lam}d_{U,V}^{\mu}=\sum\limits_{\epsilon\in\Lambda,X,Y\in M(\epsilon)}
R_{(S,T,\lam),(U,V,\mu),(X,Y,\epsilon)}d_{X,Y}^{\epsilon}.$$

For any $\lam\in\Lambda$ and $T\in M(\lam)$, set
$$e_{\lam}=\sum_{S\in
M(\lam)}C_{S,T}^{\lam}d_{T,S}^{\lam}\,\,\,\,{\rm
and}\,\,\,\,L_{\alpha}(A)=\{\sum_{\lam\in\Lambda}r_{\lam}e_{\lam}\mid
r_{\lam}\in R\}.$$

\begin{lem}\label{3.9}
With the notations as above. Then $H_{\alpha}(A)\subseteq
L_{\alpha}(A)\subseteq Z_{\alpha}(A)$. Moreover, $\dim
L_{\alpha}(A)\geq |\Lambda_0|$.
\end{lem}

\begin{proof}\,The proof of $H_{\alpha}(A)\subseteq L_{\alpha}(A)$ is similar to that of
\cite[Theorem 3.2]{L1} by using Lemma \ref{3.2}. We omit the
details here. Now let us prove $L_{\alpha}(A)\subseteq
Z_{\alpha}(A)$. Clearly, we only need to show that
$e_{\lam}\alpha^{-1}(C_{U,V}^{\mu})=C_{U,V}^{\mu}e_{\lam}$ for
arbitrary $\lam,\mu\in\Lambda$ and $U,V\in M(\mu)$.

On one hand, by the definition of $\alpha$ and Lemma \ref{3.2},
\begin{eqnarray*}
e_{\lam}\alpha^{-1}(C_{U,V}^{\mu})&=&\sum_{S\in
M(\lam)}C_{S,T}^{\lam}\alpha^{-1}(D_{T,S}^{\lam}C_{U,V}^{\mu})\\&=&\sum_{S\in
M(\lam)}\sum_{\epsilon\in\Lambda,X,Y\in
M(\epsilon)}r_{(U,V,\mu),(Y,X,\epsilon),(S,T,\lam)}C_{S,T}^{\lam}d_{X,Y}^{\epsilon}\\
&=&\sum_{S,Y\in
M(\lam)}r_{(U,V,\mu),(Y,T,\lam),(S,T,\lam)}C_{S,T}^{\lam}d_{T,Y}^{\lam}.
\end{eqnarray*}
On the other hand,
\begin{eqnarray*}
C_{U,V}^{\mu}e_{\lam}&=&\sum_{S\in
M(\lam)}\sum_{\epsilon\in\Lambda,X,Y\in
M(\epsilon)}r_{(U,V,\mu),(S,T,\lam),(X,Y,\epsilon)}C_{X,Y}^{\epsilon}d_{T,S}^{\lam}\\
&=&\sum_{S,X\in
M(\lam)}r_{(U,V,\mu),(S,T,\lam),(X,T,\lam)}C_{X,T}^{\lam}d_{T,S}^{\lam}.
\end{eqnarray*}
The proof of $\dim L_{\alpha}(A)\geq |\Lambda_0|$ is similar to
that of \cite[Proposition 3.3 (2)]{L1}.
\end{proof}

Then \cite[Theorem 3.2]{L1} becomes a corollary of Lemma
\ref{3.9}.

\begin{remark}
If we define $e'_{\lam}=\sum_{S\in
M(\lam)}D_{S,S}^{\lam}C_{S,S}^{\lam}$ and
$$L_{\alpha}(A)'=\{\sum_{\lam\in\Lambda}r_{\lam}e'_{\lam}\mid r_{\lam}\in
K\},$$ then there are analogous results on $L_{\alpha}(A)'$.
However, $L_{\alpha}(A)\neq L_{\alpha}(A)'$ in general. We will
give an example in Section 4.
\end{remark}

Certainly, another class of elements of $A$ would not be ignore,
that is, $\sum_{S\in M(\lam)}C_{S,T}^{\lam}D_{T,S}^{\lam}$.
Unfortunately, we know from Lemma \ref{3.2} that
$C_{S,T}^{\lam}D_{T,S}^{\lam}$ is more complex than
$D_{T,S}^{\lam}C_{S,T}^{\lam}$ and hence we do not know more
properties about them  so far.

\smallskip

The next aim of this section is to construct ideals of the center
for a Frobenius cellular algebra. We have the following result.

\begin{prop}\label{3.11}
Let $x\in Z_{\alpha^{-1}}(A)$. Then $\sum_{S\in
M(\lam)}C_{S,T}^{\lam}xD_{T,S}^{\lam}\in Z(A)$ for arbitrary
$\lam\in\Lambda$. Furthermore, $$Z_{\lam}(A):=\{\sum_{S\in
M(\lam)}C_{S,T}^{\lam}xD_{T,S}^{\lam}\mid x\in
Z_{\alpha^{-1}}(A)\}$$ are ideals of $Z(A)$  with
$Z_{\lam}(A)Z_{\mu}(A)=0$ if  $\lam\neq\mu$.
\end{prop}

\begin{proof}
It follows from $x\in Z_{\alpha^{-1}}(A)$ that $$\sum_{S\in
M(\lam)}C_{S,T}^{\lam}xD_{T,S}^{\lam}=e_{\lam}x.$$ Then Lemma
\ref{2.5} and \ref{3.9} imply that $e_{\lam}x\in Z(A)$. Note that
$cx, xc\in Z_{\alpha^{-1}}(A)$ for $c\in Z(A)$ and $x\in
Z_{\alpha^{-1}}(A)$. Thus  $Z_{\lam}(A)$ is an ideal of $Z(A)$.
Moreover, if $\lam\neq\mu$, then Lemma \ref{3.2} implies that
$Z_{\lam}(A)Z_{\mu}(A)=0$.
\end{proof}

\begin{question}
Denote by $Z_{\Lambda}(A)$ the central ideal generated by
$\cup_{\lam\in\Lambda} Z_{\lam}(A)$. What is the relationship
between $Z_{\Lambda}(A)$ and $H(A)$?
\end{question}

\bigskip

\section{Examples}
As is well known, all symmetric algebras are Frobenius. However,
we have rarely found examples of non-symmetric Frobenius cellular
algebras. In this section, we will give some such examples. The
first one comes from local algebras.

\begin{example}(Nakayama-Nesbitt)
Let $K$ be a field with two nonzero elements $u$, $v$ such that
$u\neq v$. Let $A\subset M_{4\times 4}(K)$ be a $K$-algebra with a
basis $$\{C_{1,1}^1=E_{14}, C_{1,1}^2=E_{13}+uE_{24},
C_{1,1}^3=E_{12}+vE_{34}, C_{1,1}^4=E.\} \eqno(\ast)$$ It is
proved that $A$ is a non-symmetric Frobenius algebra. For details,
see \cite{LT}. Moreover, the basis $(\ast)$ is clearly cellular.
Thus $A$ is a non-symmetric Frobenius cellular algebra. Note that
$A$ is symmetric if $u=v\neq 0$.
\end{example}

\begin{example}
Let $K$ be a field and $Q$ be the following quiver
$$
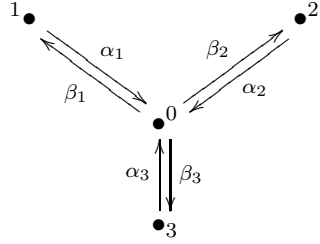
\xymatrix@C=13mm{
 ^{1}\bullet \ar@<2.0pt>[dr]^{\alpha_1} & &  \bullet^{2}\ar@<2.0pt>[dl]^(0.5){\alpha_2}\\
 &  \bullet^0\ar@<2.0pt>[ul]^(0.5){\beta_1}\ar@<2.0pt>[ur]^(0.5){\beta_2}\ar@<2.0pt>[d]^(0.5){\beta_3} & \\
 &  \bullet_{3}\ar@<2.0pt>[u]^(0.5){\alpha_3}
}$$
 with relation
$\rho$ given as follows:
$$\alpha_1\beta_1, \alpha_1\beta_2, \alpha_2\beta_1,
 \alpha_2\beta_3, \alpha_3\beta_2, \alpha_3\beta_3,
 \beta_1\alpha_1-\beta_2\alpha_2, \beta_2\alpha_2-\beta_3\alpha_3.$$
It is easy to check that $A=K(Q,\rho)$ is a Frobenius algebra with
Nakayama permutation $(13)$, that is, $A$ is a non-symmetric
Frobenius algebra. We claim that $A$ is also cellular. Here is a
cellular basis of it.
\[ \begin{matrix}

\begin{matrix} e_3 & \alpha_3\\ \beta_3 &
e_0\end{matrix} \,\,;\quad &
\begin{matrix}
e_1 & \alpha_1 & \alpha_1\beta_3\\
\beta_1 & \beta_1\alpha_1 & \beta_2\\
\alpha_3\beta_1 & \alpha_2 & e_2
 \end{matrix} \,\,;\quad & \begin{matrix} \alpha_2\beta_2
 \end{matrix}.
\end{matrix} \]

Define $\tau$ by\\
(1) $\tau(e_{1})=\cdots=\tau(e_{4})=0$;\\
(2)
$\tau(\alpha_{1}\beta_{3})=\tau(\alpha_{3}\beta_{1})=\tau(\alpha_{2}\beta_{2})=\tau(\beta_1\alpha_1)=1$,\\
(3)$\tau(\alpha_{i})=\tau(\beta_{i})=0$, $i=1,\cdots, 4$.\\

\smallskip

The right dual basis is
\[ \begin{matrix}
\begin{matrix} \alpha_3\beta_1 & \beta_1\\ \alpha_3 &
\beta_1\alpha_1\end{matrix} \,\,;\quad &
\begin{matrix}
\alpha_1\beta_3 & \beta_3 & e_3\\
\alpha_1 & e_0 & \alpha_2\\
e_1 & \beta_2 & \alpha_2\beta_2
 \end{matrix} \,\,;\quad & \begin{matrix} e_2
 \end{matrix}.
\end{matrix} \]
The left dual basis is
\[ \begin{matrix}
\begin{matrix} \alpha_1\beta_3 & \beta_3\\ \alpha_1 &
\beta_1\alpha_1\end{matrix} \,\,;\quad &
\begin{matrix}
\alpha_3\beta_1 & \beta_1 & e_1\\
\alpha_3 & e_0 & \alpha_2\\
e_3 & \beta_2 & \alpha_2\beta_2
 \end{matrix} \,\,;\quad & \begin{matrix} e_2
 \end{matrix}.
\end{matrix} \]
\end{example}

Let us give some remarks on this example.

\smallskip

1. A direct computation yields that $L_{\alpha}(A)$ is a $K$-space
of dimension 3 with a basis $\{\alpha_3\beta_1+\beta_1\alpha_1,
\alpha_1\beta_3, \alpha_2\beta_2\}$, that $L_{\alpha}(A)'$ is a
$K$-space generated by $\{3\alpha_3\beta_1, 2\alpha_1\beta_3,
\alpha_2\beta_2\}$. This implies that $L_{\alpha}(A)\neq
L_{\alpha}(A)'$. Moreover, if $K$ is of characteristic 3, then
$L_{\alpha}(A)'\subsetneqq L_{\alpha}(A)$. Note that
$H_{\alpha}(A)\subset L_{\alpha}(A)'\cap L_{\alpha}(A)$. Thus if
$K$ is of characteristic not 3, then $H_{\alpha}(A)\neq
L_{\alpha}(A)'$ and $H_{\alpha}(A)\neq L_{\alpha}(A)$.

2. By a direct computation we obtain that $i\alpha=\alpha i$.

3. For $S, T, U, V\in M(\lam)$, $d_{S,T}^{\lam}C_{U,V}^{\lam}$ may
not be 0 when $T\neq U$. For example, $d_{13}^2C_{12}^2=\alpha_1$.
Furthermore, let $\lam, \mu\in \Lambda$ with $\mu\nleq\lam$,
$d_{P,Q}^{\mu}C_{U,V}^{\lam}$ may not be 0. For example,
$d_{11}^3C_{32}^2=\alpha_2$.

4. Define a new involution $i'$ on $A$ by
\[\begin{matrix}
i(\alpha_1)=\beta_3 & & i(\alpha_2)=\alpha_2\\
i(\alpha_3)=\beta_1 & & i(e_0)=e_0\\
i(e_1)=e_3 & & i(e_2)=e_2
\end{matrix}.
\] Then the (left)
right dual basis is cellular.

\smallskip

The next example can be found in \cite{KX}. We illustrate it here
for the sake of giving some remarks.
\begin{example}
Let $K$ be a field. Let us take $\lam\in K$ with $\lam\neq 0$ and
$\lam\neq 1$. Let $$A=K<a, b, c, d>/I,$$ where $I$ is generated by
$$a^2, b^2, c^2, d^2, ab, ac, ba, bd, ca, cd, db, dc, cb-\lam bc,
ad-bc, da-bc.$$ If we define $\tau$ by
$\tau(1)=\tau(a)=\tau(b)=\tau(c)=\tau(d)=0$ and $\tau(bc)=1$ and
define an involution $i$ on $A$ to be fixing $a$ and $d$, but
interchanging $b$ and $c$,  then $A$ is a Frobenius cellular
algebra with a cellular basis
\[ \begin{matrix}
\begin{matrix} bc \end{matrix} ;&
\begin{matrix} a & b\\ c &
d\end{matrix} ; &
\begin{matrix} 1 \end{matrix}.
\end{matrix} \]
The right dual basis is
\[ \begin{matrix}
\begin{matrix} 1 \end{matrix} ;&
\begin{matrix} d & c\\ b/\lam &
a\end{matrix} ; &
\begin{matrix} bc \end{matrix}.
\end{matrix} \]
The left dual basis is
\[ \begin{matrix}
\begin{matrix} 1 \end{matrix} ;&
\begin{matrix} d & c/\lam\\ b &
a\end{matrix} ; &
\begin{matrix} bc \end{matrix}.
\end{matrix} \]
\end{example}
\begin{remarks}
1. Clearly, the matrix associated with $\alpha$ with respect to
the right dual basis is
$$\left(
\begin{array}
[c]{cccccc}%
1 & 0 & 0 & 0 & 0 & 0 \\
0 & 1 & 0 & 0 & 0 & 0 \\
0 & 0 & 1/\lambda & 0 & 0 & 0 \\
0 & 0 & 0 & \lambda & 0 & 0 \\
0 & 0 & 0 & 0 & 1 & 0 \\
0 & 0 & 0 & 0 & 0 & 1 \\
\end{array}
\right).
$$ Then the finiteness of the order of $\alpha$ is independent of
cellularity of $A$. In fact, let $K=\mathbb{C}$. If $\lam$ is not
a root of unity, then the order of $\alpha$ is infinite.
Otherwise, the order of $\alpha$ is finite.

2. The example implies that $\tau(a)=\tau(i(a))$ is not a
sufficient condition such that dual bases being cellular in
Frobenius case.

3. We have from this example that in general $Z(A)\cap
Z_{\alpha}(A)\neq \{0\}$. In fact, $bc\in Z(A)$ is clear. On the
other hand, $bc=C_{11}^1d_{11}^1\in Z_{\alpha}(A)$.

4. It is easy to check that $i\alpha\neq \alpha i$ and thus
$\{\alpha(C_{S,T}^{\lam})\mid \lam\in\Lambda, S,T\in M(\lam)\}$ is
not cellular.

5. Note that $d_{1,1}^{2}C_{1,1}^{2}=da$ and
$d_{1,2}^{2}C_{2,1}^{2}=bc/\lam$.  This implies that
$d_{S,T}^{\lam}C_{T,S}^{\lam}\neq d_{S,U}^{\lam}C_{U,S}^{\lam}$ in
general.
\end{remarks}

\bigskip

\noindent{\bf Acknowledgement} The author would like to express
his sincere thanks to Chern Institute of Mathematics of Nankai
University for the hospitality during his visit. He also
acknowledges Prof. Chengming Bai for his kind consideration and
warm help.

\end{document}